\documentclass[reqno,centertags]{amsart}
\usepackage{amssymb}
\usepackage{amsfonts}

\newtheorem{theorem}{Theorem}
\theoremstyle{plain}
\newtheorem{acknowledgement}{Acknowledgement}

\newtheorem{corollary}{Corollary}

\newtheorem{problem}{Problem}
\newtheorem{proposition}{Proposition}
\newtheorem{remark}{Remark}

\numberwithin{equation}{section}

\input{tcilatex}

\begin{document}
\title[]{Translation surfaces of coordinate finite type}
\author{Hassan Al-Zoubi}
\address{Department of Mathematics, Al-Zaytoonah University of Jordan,
P.O.Box 130\\
Amman, Jordan 11733}
\email{dr.hassanz@zuj.edu.jo}
\author{Stylianos Stamatakis}
\address{Department of Mathematics, Aristotle University of Thessaloniki}
\email{stamata@math.auth.gr}
\author{Waseem Al-Mashaleh}
\address{Department of Mathematics, Al-Zaytoonah University of Jordan}
\email{w.almashaleh@zuj.edu.jo}
\author{Mohammed Awadallah}
\address{Department of Computer Science, Al-Aqsa University}
\email{ma.awadallah@alaqsa.edu.ps}
\date{}

\begin{abstract}
We consider translation surfaces in the 3-dimensional Euclidean space which are of coordinate finite type with respect to the third fundamental form $III $, i.e., their position vector $\boldsymbol{x}$ satisfies the relation $\Delta ^{III}\boldsymbol{x}=\Lambda \boldsymbol{x,}$ where $\Lambda $ is a square matrix of order 3. We show that Sherk's minimal surface is the only translation surface satisfying $\Delta^{III}\boldsymbol{x}=\Lambda\boldsymbol{x}$.
\end{abstract}

\maketitle

\section{Introduction}

Let $M^{n}$ be a (connected) submanifold in the $m$-dimensional Euclidean space $E^{m}$. Let $\boldsymbol{x},\boldsymbol{H}$ be the position vector
field and the mean curvature field of $M^{n}$ respectively. Denote by $\Delta ^{I}$ the second Beltrami-Laplace operator corresponding to the first
fundamental form $I$ of $M^{n}$. Then, it is well known that \cite{R3}
\begin{equation}
\Delta ^{I}\boldsymbol{x}=-n\boldsymbol{H}.  \notag
\end{equation}

From this formula one can see that $M^{n}$ is a minimal submanifold if and only if all coordinate functions, restricted to $M^{n}$, are eigenfunctions
of $\Delta ^{I}$ with eigenvalue $\lambda =0$. Moreover in \cite{R15} T.Takahashi showed that a submanifold $M^{n}$ for which $\Delta^{I}\boldsymbol{x} =\lambda\boldsymbol{x}$, i.e. for which all coordinate functions are eigenfunctions of $\Delta ^{I}$ with the same eigenvalue $\lambda$, is either minimal in $E^{m}$ with eigenvalue $\lambda=0$ or minimal in a hypersphere of $E^{m}$ with eigenvalue $\lambda>0$. Although the class of finite type submanifolds
in an arbitrary dimensional Euclidean spaces is very large, very little is known about surfaces of finite type in the Euclidean 3-space $E^{3}$.
Actually, so far, the only known surfaces of finite type corresponding to the first fundamental form in the Euclidean 3-space are the minimal
surfaces, the circular cylinders and the spheres. So in \cite{R5} B.-Y. Chen mentions the following problem

\begin{problem}
\label{(1)}Determine all surfaces of finite Chen $I$-type in $E^{3}$.
\end{problem}

In order to provide an answer to the above problem, important families of surfaces were studied by different authors by proving that finite type ruled
surfaces, finite type quadrics, finite type tubes, finite type cyclides of Dupin and finite type spiral surfaces are surfaces of the only known
examples in $E^{3}$. However, for another classical families of surfaces, such as surfaces of revolution, translation surfaces as well as helicoidal
surfaces, the classification of its finite type surfaces is not known yet. For a more details, the reader can refer to \cite{R6}.

Later in \cite{R11} O. Garay generalized T. Takahashi's condition studied surfaces in $E^{3}$ for which all coordinate functions $\left(
x_{1},x_{2},x_{3}\right)$ of $\boldsymbol{x}$ satisfy $\Delta^{I}\boldsymbol{x}_{i} = \lambda_{i}x_{i}, i = 1,2,3$, not necessarily with the same
eigenvalue. Another generalization was studied in \cite{R9} for which surfaces in $E^{3}$ satisfy the condition $\Delta^{I}\boldsymbol{x}= A\boldsymbol{x} + B$ $(\ddag)$ where $A \in\mathbb{Re}^{3\times3} ;B \in\mathbb{Re}^{3\times1}$. It was shown that a surface $S$ in $E^{3}$ satisfies $(\ddag)$ if and only if it is an open part of a minimal surface, a sphere, or a circular cylinder. Surfaces satisfying $(\ddag)$ are said to be of coordinate finite type.

In the thematic circle of the surfaces of finite type in the Euclidean space in $E^{3}$, S. Stamatakis and H. Al-Zoubi in \cite{R13} restored attention
to this theme by introducing the notion of surfaces of finite type corresponding to the second or third fundamental forms of $S$ in the
following way: A surface $S$ is said to be of finite type corresponding to the fundamental form $J$, or briefly of finite $J$-type, $J=II,III$, if the
position vector $\boldsymbol{x}$ of $S$ can be written as a finite sum of nonconstant eigenvectors of the operator $\Delta ^{J}$, that is if
\begin{equation}
\boldsymbol{x}=\boldsymbol{x}_{0}+\sum_{i=1}^{k}\boldsymbol{x}_{i},\ \ \ \ \
\ \Delta ^{J}\boldsymbol{x}_{i}=\lambda _{i}\boldsymbol{x}_{i},\ \ \
i=1,...,k,  \label{r1}
\end{equation}%
where $\boldsymbol{x}_{0}$ is a fixed vector and $\boldsymbol{x}_{1},...,%
\boldsymbol{x}_{k}$ are nonconstant maps such that $\Delta ^{J}\boldsymbol{\
x}_{i}=\lambda _{i}\boldsymbol{x}_{i},i=1,...,k$. If, in particular, all
eigenvalues $\lambda _{1},\lambda _{2},...,\lambda _{k}$ are mutually
distinct, then $S$ is said to be of $J$-type $k$, otherwise $S$ is said to
be of infinite type. When $\lambda _{i}=0$ for some \emph{i} = 1,..., \emph{k%
}, then $S$ is said to be of null $J$-type $k$.

In general when $S$ is of finite type $k$, it follows from (\ref{r1}) that there exist a monic polynomial, say $R(x)\neq 0,$ such that $R(\Delta ^{J})(\boldsymbol{x}-\boldsymbol{c})=\mathbf{0}.$ If we suppose that $R(x)=x^{k}+\sigma _{1}x^{k-1}+...+\sigma _{k-1}x+\sigma _{k},$ then coefficients $\sigma _{i}$\ are given by

\begin{eqnarray}
\sigma _{1} &=&-(\lambda _{1}+\lambda _{2}+...+\lambda _{k}),  \notag \\
\sigma _{2} &=&(\lambda _{1}\lambda _{2}+\lambda _{1}\lambda_{3}+...+\lambda
_{1}\lambda _{k}+\lambda _{2}\lambda _{3}+...+\lambda _{2}\lambda
_{k}+...+\lambda _{k-1}\lambda _{k}),  \notag \\
\sigma _{3} &=&-(\lambda _{1}\lambda _{2}\lambda _{3}+...+\lambda
_{k-2}\lambda _{k-1}\lambda _{k}),  \notag \\
&&.............................................  \notag \\
\sigma _{k} &=&(-1)^{k}\lambda _{1}\lambda _{2}...\lambda _{k}.  \notag
\end{eqnarray}

Therefore the position vector $\boldsymbol{x}$ satisfies the following equation \cite{R3}

\begin{equation*}
(\Delta ^{J})^{k}\boldsymbol{x}+\sigma _{1}(\Delta ^{J})^{k-1}\boldsymbol{x}%
+...+\sigma _{k}(\boldsymbol{x}-\boldsymbol{c})=\boldsymbol{0}.
\end{equation*}

In this paper we will pay attention to surfaces of finite $III$-type. Firstly, we will establish a formula for $\Delta ^{III}\boldsymbol{x}$ by
using Cartan's method of the moving frame. Further, we continue our study by proving finite coordinate type surfaces for an important class of surfaces,
namely, translation surfaces in $E^{3}$.

\section{Preliminaries}

Let $S$ be a (connected) surface in the Euclidean 3-space $E^{3}$, whose
Gaussian curvature $K$ never vanishes. Let $\wp =\{\boldsymbol{\varepsilon
_{1}}(u,v),\boldsymbol{\varepsilon _{2}}(u,v),\boldsymbol{\varepsilon _{3}}%
(u,v)\}$ is a moving frame of $S$, $\boldsymbol{\varepsilon _{3}}=%
\boldsymbol{n}$ is the Gauss map of $S$ and $\det (\boldsymbol{\varepsilon
_{1}},\boldsymbol{\varepsilon _{2}},\boldsymbol{\varepsilon _{3}})=1$. Then
it is well known that there are linear differential forms $\omega
_{1},\omega _{2},\omega _{31},\omega _{32}$ and $\omega _{12}$, such that
\cite{R10}

\begin{equation*}
d\boldsymbol{x}=\omega _{1}\boldsymbol{\varepsilon _{1}}+\omega _{2}%
\boldsymbol{\varepsilon _{2}},\ \ \ d\boldsymbol{n}=\omega _{31}\boldsymbol{%
\varepsilon _{1}}+\omega _{32}\boldsymbol{\varepsilon _{2}},
\end{equation*}

\begin{center}
\begin{equation*}
d\boldsymbol{\varepsilon _{1}}=\omega _{12}\boldsymbol{\varepsilon _{2}}%
-\omega _{31}\boldsymbol{\varepsilon _{3}},\ \ \ d\boldsymbol{\varepsilon
_{2}}=-\omega _{12}\boldsymbol{\varepsilon _{1}}-\omega _{32}\boldsymbol{%
\varepsilon _{3}},
\end{equation*}
\end{center}

and functions $a,b,c,q_{1},q_{2}$ of $u,v$ such that

\begin{equation*}
\omega _{31}=-a\omega _{1}-b\omega _{2},\ \ \ \omega _{32}=-b\omega
_{1}-c\omega _{2},\ \ \ \omega _{12}=q_{1}\omega _{1}+q_{2}\omega _{2}.
\end{equation*}

We can choose the moving frame of $S$, such that the vectors $\boldsymbol{%
\varepsilon_{1}},\boldsymbol{\varepsilon_{2}}$ are the principle directions
of $S$. Then $a$, $c$ are the principle curvatures of $S$ and $b=0$, so the
differential forms $\omega_{1}$ and $\omega_{2}$ become

\begin{equation}
\omega_{1} =-\frac{1}{a}\omega_{31}, \ \ \ \ \omega_{2} =-\frac{1}{c}%
\omega_{32}.  \notag
\end{equation}
The Gauss and mean curvature are respectively

\begin{equation}
K= ac, \ \ \ \ H=\frac{a+c}{2}.  \notag
\end{equation}

Let $\nabla _{1}f,\nabla _{2}f$ be the derivatives of Pfaff of a function $%
f(u,v)\in C^{1}$ along the curves $\omega _{2}=0,\omega _{1}=0$
respectively. Then we have the following well known relations \cite{R2}

\begin{eqnarray*}
\nabla _{1}\boldsymbol{x} &=&\boldsymbol{\varepsilon _{1}},\ \ \ \nabla _{2}%
\boldsymbol{x}=\boldsymbol{\varepsilon _{2}}, \\
\nabla _{1}\boldsymbol{\varepsilon _{1}} &=&q_{1}\boldsymbol{\varepsilon _{2}%
}+a\boldsymbol{n},\ \ \ \nabla _{2}\boldsymbol{\varepsilon _{1}}=q_{2}%
\boldsymbol{\varepsilon _{2}}, \\
\nabla _{1}\boldsymbol{\varepsilon _{2}} &=&-q_{1}\boldsymbol{\varepsilon
_{1}},\ \ \ \ \nabla _{2}\boldsymbol{\varepsilon _{2}}=-q_{2}\boldsymbol{%
\varepsilon _{1}}+c\boldsymbol{n}, \\
\nabla _{1}\boldsymbol{n} &=&-a\boldsymbol{\varepsilon _{1}},\ \ \ \nabla
_{2}\boldsymbol{n}=-c\boldsymbol{\varepsilon _{2}},
\end{eqnarray*}

We denote by $\widetilde{\nabla}_{1}f$ and $\widetilde{\nabla}_{2}f$ the derivatives of Pfaff of $f$ along the curves $\omega_{32}=0, \omega_{31}=0$
respectively. One finds

\begin{equation}
\widetilde{\nabla}_{1}f =-\frac{1}{a}\nabla_{1}f,\ \ \ \widetilde{\nabla}_{2}f =-\frac{1}{c}\nabla_{2}f.  \notag
\end{equation}

It follows that

\begin{equation}
\widetilde{\nabla }_{1}\boldsymbol{x}=-\frac{1}{a}\boldsymbol{\varepsilon
_{1}},\ \ \ \widetilde{\nabla }_{2}\boldsymbol{x}=-\frac{1}{c}\boldsymbol{%
\varepsilon _{2}},  \label{r3}
\end{equation}%
\begin{equation}
\widetilde{\nabla }_{1}\boldsymbol{\varepsilon _{1}}=p_{1}\boldsymbol{%
\varepsilon _{2}}-\boldsymbol{n},\ \ \ \widetilde{\nabla }_{2}\boldsymbol{%
\varepsilon _{1}}=p_{2}\boldsymbol{\varepsilon _{2}},  \label{r4}
\end{equation}%
\begin{equation}
\widetilde{\nabla }_{1}\boldsymbol{\varepsilon _{2}}=-p_{1}\boldsymbol{%
\varepsilon _{1}},\ \ \ \ \widetilde{\nabla }_{2}\boldsymbol{\varepsilon _{2}%
}=-p_{2}\boldsymbol{\varepsilon _{1}}-\boldsymbol{n},  \label{r5}
\end{equation}%
\begin{equation}
\widetilde{\nabla }_{1}\boldsymbol{n}=\boldsymbol{\varepsilon _{1}},\ \ \
\widetilde{\nabla }_{2}\boldsymbol{n}=\boldsymbol{\varepsilon _{2}},  \notag
\end{equation}%
where $p_{1}=-\frac{1}{a}q_{1},p_{2}=-\frac{1}{c}q_{2}$ are the geodesic
curvatures of the spherical curves $\omega _{32}=0$ and $\omega _{31}=0$
respectively. The Mainardi-Codazzi equations have the following forms

\begin{equation}
\widetilde{\nabla }_{1}\frac{1}{c}=p_{2}\Big(\frac{1}{a}-\frac{1}{c}\Big),\
\ \ \widetilde{\nabla }_{2}\frac{1}{a}=p_{1}\Big(\frac{1}{a}-\frac{1}{c}\Big)%
.  \label{r6}
\end{equation}

Let $f$ be a sufficient differentiable function on $S$. Then the second differential parameter of Beltrami corresponding to the fundamental form $%
III $ of $S$ is defined by %\cite{R12}
\begin{equation}  \label{777}
\Delta ^{III}f= -\widetilde{\nabla}_{1}\widetilde{\nabla}_{1}f-\widetilde{%
\nabla}_{2}\widetilde{\nabla}_{2}f-p_{2}\widetilde{\nabla}_{1}f+p_{1}%
\widetilde{\nabla}_{2}f.
\end{equation}

Applying (\ref{777}) to the position vector $\boldsymbol{x}$, gives

\begin{equation}
\Delta ^{III}\boldsymbol{x}= -\widetilde{\nabla}_{1}\widetilde{\nabla}_{1}%
\boldsymbol{x}-\widetilde{\nabla}_{2}\widetilde{\nabla}_{2}\boldsymbol{x}%
-p_{2}\widetilde{\nabla}_{1}\boldsymbol{x}+p_{1}\widetilde{\nabla}_{2}%
\boldsymbol{x}.  \notag
\end{equation}

From (\ref{r3}) we obtain

\begin{equation}
\Delta ^{III}\boldsymbol{x}=\widetilde{\nabla }_{1}\Big(\frac{1}{a}%
\boldsymbol{\varepsilon _{1}}\Big)+\widetilde{\nabla }_{2}\Big(\frac{1}{c}%
\boldsymbol{\varepsilon _{2}}\Big)+\frac{1}{a}p_{2}\boldsymbol{\varepsilon
_{1}}-\frac{1}{c}p_{1}\boldsymbol{\varepsilon _{2}}.  \label{r8}
\end{equation}

Using (\ref{r4}) and (\ref{r5}), equation (\ref{r8}) becomes

\begin{equation}
\Delta ^{III}\boldsymbol{x}=\Big(\widetilde{\nabla }_{1}\frac{1}{a}\Big)%
\boldsymbol{\varepsilon _{1}}+\Big(\frac{1}{a}-\frac{1}{c}\Big)p_{2}%
\boldsymbol{\varepsilon _{1}}+\Big(\widetilde{\nabla }_{2}\frac{1}{c}\Big)%
\boldsymbol{\varepsilon _{2}}+\Big(\frac{1}{a}-\frac{1}{c}\Big)p_{1}%
\boldsymbol{\varepsilon _{2}}-\Big(\frac{1}{a}+\frac{1}{c}\Big)\boldsymbol{n}%
.  \label{r9}
\end{equation}

Taking into account the Mainardi-Codazzi equations (\ref{r6}), so equation (%
\ref{r9}) reduces to

\begin{equation*}
\Delta ^{III}\boldsymbol{x}=\Big(\widetilde{\nabla }_{1}\big(\frac{1}{a}+%
\frac{1}{c}\big)\Big)\boldsymbol{\varepsilon _{1}}+\Big(\widetilde{\nabla }%
_{2}\big(\frac{1}{a}+\frac{1}{c}\big)\Big)\boldsymbol{\varepsilon _{2}}-\big(%
\frac{1}{a}+\frac{1}{c}\big)\boldsymbol{n}
\end{equation*}%
or equivalently, (see \cite{R13})

\begin{equation}
\Delta ^{III}\boldsymbol{x}=grad^{III}\Big(\frac{2H}{K}\Big)-\Big(\frac{2H}{K%
}\Big)\boldsymbol{n}.  \label{r11}
\end{equation}

\begin{remark}
S. Stamatakis and H. Al-zoubi proved in \cite{R13} relation (\ref{r11}) by
using tensors calculus.
\end{remark}

From (\ref{r11}) we obtain the following results which were proved in \cite%
{R13}.

\begin{theorem}
\label{T1} A surface $S$ in $E^{3}$ is of 0-type 1 corresponding to the third fundamental form if and only if $S$ is minimal.
\end{theorem}

\begin{theorem}
\label{T2} A surface $S$ in $E^{3}$ is of $III$-type 1 if and only if $S$ is part of a sphere.
\end{theorem}

\begin{corollary}
\label{C1.1} The Gauss map of every surface $S$ in $E^{3}$ is of $III$-type 1. The corresponding eigenvalue is $\lambda=2$.
\end{corollary}

Up to now, the only known surfaces of finite $III$-type in $E^{3}$ are parts of spheres, minimal surfaces and parallel surfaces to minimal ones
which are of null $III$-type 2. So the following question seems to be interesting:

\begin{problem}
\label{p1}Other than the surfaces mentioned above, which surfaces in $E^{3}$ are of finite $III$-type?
\end{problem}

Another generalization of the above problem is to study surfaces in $E^{3}$ whose position vector $\boldsymbol{x}$ satisfies
\begin{equation}
\Delta ^{III}\boldsymbol{x}=\Lambda \boldsymbol{x,}  \label{r12}
\end{equation}%
where $\Lambda \in \mathbb{Re}^{3\times 3}$.

From this point of view, we also pose the following problem

\begin{problem}
\label{p2}Classify all surfaces in $E^{3}$ whose position vector $\boldsymbol{x}$ satisfies relation (\ref{r12}).
\end{problem}

Concerning this problem, in \cite{R14} S. Stamatakis and H. Al-Zoubi studied the class of surfaces of revolution and they proved that:

\begin{theorem}
\label{TC} A surface of revolution $S$ in $%
%TCIMACRO{\U{211d} }%
%BeginExpansion
\mathbb{R}
%EndExpansion
^{3},$ satisfies (\ref{r12}), if and only if S is a catenoid or a part of a
sphere.
\end{theorem}

Recently, the same authors in \cite{R1} studied the class of ruled surfaces and the class of quadric surfaces. More precisely, they proved that

\begin{theorem}
\label{TA} The only \textit{ruled surfaces in the 3-dimensional Euclidean space that satisfies (\ref{r12}), are the helicoids.}
\end{theorem}

\begin{theorem}
\label{TB} The only quadric surfaces in the 3-dimensional \textit{Euclidean } space that satisfies (\ref{r12}), are the spheres.
\end{theorem}

The present paper contributes to the solution of problem \ref{p2}, for the class of translation surfaces in the Euclidean 3-space $E^{3}$, meanwhile
problem \ref{p1} is still unsolved for this family of surfaces.

\section{Translation surfaces}

Let $S$ : $M^{2}\rightarrow E^{3}$ be a translation surface in the Euclidean 3-space. Then the position vector $\boldsymbol{x}$ of $S$ is given by \cite{R25}

\begin{equation}
S:\boldsymbol{x}(s,t)=\left\{ s,\ t,\ \widetilde{f}(s)+\widetilde{h}%
(t)\right\} ,\ \ \ \ \ \ \ \ (s,t)\in D\subset M^{2},  \label{1}
\end{equation}%
where $\widetilde{f}(s)$ and $\widetilde{h}(t)$ are two sufficiently differentiable functions on $S$. We put

\begin{equation*}
f:=\frac{d\widetilde{f}}{ds},\ \ \ \ h:=\frac{d\widetilde{h}}{dt}.
\end{equation*}

It is easily verified that the first and the second fundamental forms of $S$ are given by
\begin{align*}
I& =(1+f^{2})ds^{2}+2fhdsdt+(1+h^{2})dt^{2}, \\
II& =\frac{f_{s}}{\sqrt{\mu }}ds^{2}+\frac{h_{t}}{\sqrt{\mu }}dt^{2},
\end{align*}%
where

\begin{equation*}
f_{s}:=\frac{df}{ds},\ \ \ \ h_{t}:=\frac{dh}{dt}
\end{equation*}%
and $\mu :=\det (g_{ij})=1+f^{2}+h^{2}.$

The Gauss and mean curvatures of $S,$ are respectively

\begin{equation*}
K=\frac{f_{s}h_{t}}{\mu ^{2}},
\end{equation*}

\begin{equation}
2H=\frac{(1+f^{2})h_{t}+(1+h^{2})f_{s}}{\mu \sqrt{\mu }}.  \label{2}
\end{equation}

Since $S$ does not contain parabolic points, so $f_{s}\neq 0$ and $h_{t}\neq0$ for each $(s,t)\in D.$

Our main result is the following

\begin{theorem}
\label{T2.1} The only\textit{\ translation surface in the 3-dimensional Euclidean space that satisfies (\ref{r12}) is Scherk's surface.}
\end{theorem}

Our discussion is local, which means that we show in fact that any open part of a translation surface satisfies (\ref{r12}), if it is minimal.

Before we start proving our main result we prove the following

\begin{proposition}
\label{C2.1} A \textit{minimal translation surface with parametric represantation (\ref{1}) has the form }%
\begin{equation}
S:\boldsymbol{x}(s,t)=\left\{ s,\ t,\ c-\frac{1}{a_{1}}\ln \left\vert \cos
(a_{1}s+a_{2})\right\vert +\frac{1}{a_{1}}\ln \left\vert \cos
(-a_{1}t+b_{2})\right\vert \right\} ,  \label{3}
\end{equation}%
\begin{equation*}
(s,t)\in D\subset M^{2},
\end{equation*}

where $a_{1},a_{2},b_{2}$ and $c$ are integration constants.\footnote{This proposition is due to H. F. Scherk. For a more details, the reader is referred to \cite{R17}.}
\end{proposition}

\begin{proof}
From (\ref{2}) we obtain that $S$ is minimal if and only if

\begin{equation}
\frac{1+f^{2}}{f_{s}}+\frac{1+h^{2}}{h_{t}}=0.  \label{e1}
\end{equation}

Since the first factor of this equation depends only on the parameter $s$ and the second factor depends only on the parameter $t,$ (\ref{e1}) implies
that $\frac{1+f^{2}}{f_{s}}=a_{1}$ and $\frac{1+h^{2}}{h_{t}}=b_{1},$ where $a_{1},$ $b_{1}$ are constants and $a_{1}+$ $b_{1}=0.$

Integrating each factor with respect to the corresponding parameter, we obtain

\begin{equation*}
f(s)=\tan (a_{1}s+a_{2}),\ \ \ \ \ \ h(t)=\tan (b_{1}t+b_{2}).
\end{equation*}

where $a_{2},$ $b_{2}$ are integration constants. Integrating again, it then follows that

\begin{equation*}
\widetilde{f}(s)=d_{1}-\frac{1}{a_{1}}\ln \left\vert \cos
(a_{1}s+a_{2})\right\vert ,\ \ \ \ \widetilde{h}(t)=d_{2}-\frac{1}{b_{1}}\ln
\left\vert \cos (b_{1}t+b_{2})\right\vert .
\end{equation*}

where $d_{1},$ $d_{2}$ are constants. Putting $b_{1}=-a_{1}$ and $d_{1}+d_{2}=c$. Thus we have\textit{\ }(\ref{3}).
\end{proof}

\section{Proof of Theorem \protect\ref{T2.1}}

As in section 2, a parametric represantation of a translation surface $S$ in $E^{3}$ is

\begin{equation*}
S:\boldsymbol{x}(s,t)=\left\{ s,\ t,\ \widetilde{f}(s)+\widetilde{h}%
(t)\right\} ,\ \ \ \ \ \ \ \ (s,t)\in D\subset M^{2},
\end{equation*}%
where $\widetilde{f}(s),\widetilde{\ h}(t)$ are two sufficiently differentiable functions on $S$. For convenience, we put

\begin{equation*}
M(s)=\frac{(1+f^{2})f_{ss}}{f_{s}^{3}}-\frac{2f}{f_{s}},\ \ \ \ N(t)=\frac{%
(1+h^{2})h_{tt}}{h_{t}^{3}}-\frac{2h}{h_{t}}.
\end{equation*}

Then the Beltrami operator with respect to the third fundamental form, after a lengthy computation, can be expressed as follows \cite{R0}

\begin{equation}
\Delta ^{III}=\left( -\frac{1+f^{2}}{f_{s}^{2}}\frac{\partial ^{2}}{\partial
s^{2}}-\frac{1+h^{2}}{h_{t}^{2}}\frac{\partial ^{2}}{\partial t^{2}}-\frac{%
2fh}{f_{s}h_{t}}\frac{\partial ^{2}}{\partial s\partial t}+M(s)\frac{%
\partial }{\partial s}+N(t)\frac{\partial }{\partial t}\right)\mu .
\label{eq 4}
\end{equation}

For later use, we remark that

\begin{equation}
-\frac{d}{ds}\left( \frac{1+f^{2}}{f_{s}}\right) =M(s)f_{s},  \label{4}
\end{equation}

\begin{equation}
-\frac{d}{dt}\left( \frac{1+h^{2}}{h_{t}}\right) =N(t)h_{t}.  \label{5}
\end{equation}

Applying (\ref{eq 4}) on the coordinate functions $x_{i},i=1,2,3$ of the position vector $\boldsymbol{x}$ of $S,$ we have

\begin{equation*}
\Delta ^{III}x_{1}=\Delta ^{III}s=M(s)\mu ,
\end{equation*}

\begin{equation*}
\Delta ^{III}x_{2}=\Delta ^{III}t=N(t)\mu ,
\end{equation*}

\begin{equation*}
\Delta ^{III}x_{3}=\Delta ^{III}(\widetilde{f}+\widetilde{h})=\left(
M(s)f+N(t)h-\frac{1+f^{2}}{f_{s}}-\frac{1+h^{2}}{h_{t}}\right) \mu .
\end{equation*}

We denote by $\lambda _{ij},i,j=1,2,3$ the entries of the matrix $\Lambda .\
$Using the last three equations and condition (\ref{r12}) one finds%
\begin{equation}
M(s)\mu =\lambda _{11}s+\lambda _{12}t+\lambda _{13}(\widetilde{f}+%
\widetilde{h}),  \label{eq 5}
\end{equation}

\begin{equation}
N(t)\mu =\lambda _{21}s+\lambda _{22}t+\lambda _{23}(\widetilde{f}+%
\widetilde{h}),  \label{eq 6}
\end{equation}

\begin{equation}
\left( M(s)f+N(t)h-\frac{1+f^{2}}{f_{s}}-\frac{1+h^{2}}{h_{t}}\right) \mu
=\lambda _{31}s+\lambda _{32}t+\lambda _{33}(\widetilde{f}+\widetilde{h}).
\label{eq 7}
\end{equation}

Differentiating (\ref{eq 5}) with respect to $t,$ we find

\begin{equation*}
2Mhh_{t}=\lambda _{12}+\lambda _{13}h.
\end{equation*}

Hence, $M(s)$ must be constant.
Similarly, if we take the derivative of (\ref{eq 6}) with respect to $s$ we
obtain that $N(t)$ must be constant. If we put $M(s)=c_{1}$ and $N(t)=c_{2}$%
, where $c_{1}$ and $c_{2}$ are constants, then equations (\ref{eq 5}), (\ref%
{eq 6}) and (\ref{eq 7}) become

\begin{equation}
c_{1}\mu =\lambda _{11}s+\lambda _{12}t+\lambda _{13}(\widetilde{f}+%
\widetilde{h}),  \label{eq 8}
\end{equation}

\begin{equation}
c_{2}\mu =\lambda _{21}s+\lambda _{22}t+\lambda _{23}(\widetilde{f}+%
\widetilde{h}),  \label{eq 9}
\end{equation}

\begin{equation}
\left( c_{1}f+c_{2}h-\frac{1+f^{2}}{f_{s}}-\frac{1+h^{2}}{h_{t}}\right) \mu
=\lambda _{31}s+\lambda _{32}t+\lambda _{33}(\widetilde{f}+\widetilde{h}).
\label{eq 10}
\end{equation}

Differentiating (\ref{eq 10}) with respect to $s,$ and taking into account (%
\ref{4}), we have

\begin{equation*}
2c_{1}f_{s}\mu +2\left( c_{1}f+c_{2}h-\frac{1+f^{2}}{f_{s}}-\frac{1+h^{2}}{%
h_{t}}\right) ff_{s}=\lambda _{31}+\lambda _{33}f.
\end{equation*}

Another derivation with respect to $t,$ and taking into account (\ref{5}),
gives

\begin{equation*}
4f_{s}h_{t}(c_{1}f+c_{2}h)=0.
\end{equation*}

Since the first factor can not be zero, we conclude that $c_{1}f+c_{2}h=0,$
and hence $c_{1}=c_{2}=0.$

From (\ref{eq 8}) and (\ref{eq 9}) we find that $\lambda
_{ij}=0,i=1,2,j=1,2,3,$ and (\ref{eq 10}) becomes

\begin{equation}
-\left( \frac{1+f^{2}}{f_{s}}+\frac{1+h^{2}}{h_{t}}\right) \mu =\lambda
_{31}s+\lambda _{32}t+\lambda _{33}(\widetilde{f}+\widetilde{h}).
\label{eq 11}
\end{equation}

Also, the system of equations (\ref{4}) and (\ref{5}) reduces to

\begin{equation}
\frac{d}{ds}\left( \frac{1+f^{2}}{f_{s}}\right) =0,  \label{eq 12}
\end{equation}

\begin{equation}
\frac{d}{dt}\left( \frac{1+h^{2}}{h_{t}}\right) =0.  \label{eq 13}
\end{equation}

From (\ref{eq 12}) and (\ref{eq 13}) we obtain

\begin{equation}
\frac{1+f^{2}}{f_{s}}=\frac{1}{a_{1}},  \label{eq 14}
\end{equation}

\begin{equation}
\frac{1+h^{2}}{h_{t}}=\frac{1}{b_{1}},  \label{eq 15}
\end{equation}%
where $\frac{1}{a_{1}}$ and $\frac{1}{b_{1}}$ are constants.

Consequently, equation (\ref{eq 11}) becomes

\begin{equation}
-\left( \frac{1}{a_{1}}+\frac{1}{b_{1}}\right) \mu =\lambda _{31}s+\lambda
_{32}t+\lambda _{33}(\widetilde{f}+\widetilde{h}).  \label{e2}
\end{equation}

By taking the derivative of (\ref{e2}) with respect to $s$, we find that

\begin{equation}
-2\left( \frac{1}{a_{1}}+\frac{1}{b_{1}}\right) ff_{s}=\lambda _{31}+\lambda
_{33}f,  \label{eq 16}
\end{equation}

Another integration of (\ref{eq 14}) and (\ref{eq 15}) gives

\begin{equation*}
f(s)=\tan (a_{1}s+a_{2}),\ \ \ \ \ \ h(t)=\tan (b_{1}t+b_{2}),
\end{equation*}%
where $a_{2},$ $b_{2}$ are integration constants.

Inserting $f(s)=\tan (a_{1}s+a_{2})$ and its derivation $f_{s}$ in (\ref{eq 16}), we get

\begin{equation*}
-2a_{1}\left( \frac{1}{a_{1}}+\frac{1}{b_{1}}\right) \frac{\sin
(a_{1}s+a_{2})}{\cos ^{3}(a_{1}s+a_{2})}=\lambda _{31}+\lambda _{33}\frac{%
\sin (a_{1}s+a_{2})}{\cos (a_{1}s+a_{2})}
\end{equation*}%
or, equivalently

\begin{equation*}
\lambda _{31}\cos ^{3}(a_{1}s+a_{2})+\lambda _{33}\sin (a_{1}s+a_{2})\cos
^{2}(a_{1}s+a_{2})
\end{equation*}

\begin{equation}
\ \ \ \ \ \ \ \ \ \ \ \ \ \ \ \ \ \ \ \ \ \ \ \ \ \ +2a_{1}\left( \frac{1}{%
a_{1}}+\frac{1}{b_{1}}\right) \sin (a_{1}s+a_{2})=0.  \label{eq 17}
\end{equation}

Notice that $a_{1}\neq 0,$ otherwise $f_{s}=0.$ Putting $z=\cos
(a_{1}s+a_{2}),$ (\ref{eq 17}) becomes

\begin{equation*}
\lambda _{31}z^{3}+\lambda _{33}z^{2}\sqrt{1-z^{2}}+2a_{1}\left( \frac{1}{%
a_{1}}+\frac{1}{b_{1}}\right) \sqrt{1-z^{2}}=0.
\end{equation*}

Since the functions $z^{3}$ and $z^{j}\sqrt{1-z^{2}}$, $0\leq j\leq 2$ are linearly independent, we have

\begin{equation*}
\lambda _{31}=\lambda _{33}=0,
\end{equation*}

\begin{equation}
\frac{1}{a_{1}}+\frac{1}{b_{1}}=0.  \label{eq 18}
\end{equation}

Finally, from (\ref{e2}) and the last two equations we find that $\lambda_{32}=0$. Also, from (\ref{eq 18}) we get $a_{1}+$ $b_{1}=0.$ Integrating $%
f(s)=\tan (a_{1}s+a_{2})$ with respect to $s,$ $h(t)=\tan (b_{1}t+b_{2})$ with respect to $t$ and using proposition \ref{C2.1}, we find that the only
translation surfaces that satisfies (\ref{r12}), are the minimals with corresponding matrix $\Lambda $ the null matrix. This ends the proof of our theorem.

\begin{acknowledgement}
The authors would like to express their thanks to the referee for his useful remarks.
\end{acknowledgement}

\end{document}